\documentclass[11pt,a4paper]{article}
\usepackage[T1]{fontenc}
\usepackage{url}
\usepackage{amssymb}
\usepackage{amsmath}
\usepackage{mathrsfs}
\usepackage{hyperref}
\usepackage{nicefrac}
\usepackage{upgreek}
\usepackage[english]{babel}
\oddsidemargin=\evensidemargin
 
\newtheorem{theorem}{Theorem}[section]

\newtheorem{lemma}[theorem]{Lemma}

\newtheorem{remark}[theorem]{Remark}

\numberwithin{equation}{section} 

\newcommand{\Rot}{\operatorname{\mathbf{curl}}}
\newcommand{\Div}{\operatorname{\mathrm{div}}}

\newcommand  {\R}{{\mathbb R}}
\newcommand {\Id}{\mathrm {I}}

\newcommand  {\LL}{\boldsymbol{L}}
\newcommand  {\HH}{\boldsymbol{H}}

\newcommand{\hh}{{\boldsymbol{ h}}}
\newcommand  {\nn}{\boldsymbol{ n}}
\newcommand  {\uu}{\boldsymbol{ u}}
\newcommand  {\vv}{\boldsymbol{ v}} 

\newenvironment{proof}{\begin{trivlist}
                       \item[]{\sc Proof. }}{\hfill $\blacksquare$
                     \end{trivlist}} 
  \newcommand{\transposee}[1]{{\vphantom{#1}}{#1}^{\sf T}}    
\begin{document}

\title{On the Fr\'echet derivative in elastic obstacle scattering}
\author{
  Fr\'ed\'erique Le Lou\"er
  \thanks{ Institut f\"ur Numerische und Andgewandte Mathematik, Universit\"at G\"ottingen, 37083
   G\"ottingen, Germany,  f.lelouer@math.uni-goettingen.de }
  }
  \date{}

\maketitle

\begin{abstract}   In this paper, we investigate the existence and characterizations of the Fr\'echet derivatives of the solution to time-harmonic elastic scattering problems with respect to the boundary of the obstacle. Our analysis is based on a technique - the factorization of the difference of the far-field pattern for two different scatterers - introduced by Kress and Pa\"{i}varinta to establish Fr\'echet differentiability in acoustic scattering. For the Dirichlet boundary condition an alternative proof of a differentiability result due to Charalambopoulos is provided and new results are proven for  the Neumann and impedance exterior boundary value problems.
\end{abstract}      


\section{Introduction}
The  inverse obstacle scattering problem for time harmonic waves is to determine the shape of the boundary and the location of a scatterer from far field measurements of the total wave.  This problem is of practical interest in some important fields of applied physics, as for example non destructive testing in linear elasticity.
Although such an inverse problem is theoretically difficult to solve since it is ill-posed and nonlinear, one can apply numerical methods to recover an approximate solution. The use of regularized iterative methods via first order linearization requires the Fr\'echet differentiability analysis of the far-field pattern  of the solution to the forward problem with respect to  the boundary of the scatterer.  An explicit form of the first  derivative is needed in view of its implementation in iterative algorithms.

In acoustic scattering,  Fr\'echet differentiability with characterizations of the de\-rivative as the far-field pattern of  the radiating solution to a new exterior boundary value problem were investigated by Hettlich \cite{Hettlich} and Kirsch \cite{Kirsch}   via variational methods,  and by Hohage \cite{Hohage}  via the implicit function theorem.  These characterizations allow the numerical implementation of the derivatives from the knowledge of the boundary values of the scattered wave only (see \cite{HohageHarbrecht,Hohage,Kirsch}).
By the use of  boundary integral equation methods, one can express the far field  pattern of the solution to scattering problems  in terms of products of boundary integral operators with singular Schwarz kernels. The Fr\'echet differentiability analysis of the  far field was then developped
by Potthast \cite{Potthast2,Potthast1}  for the Dirichlet and Neumann acoustic  problems via the Fr\'echet differentiability analysis of the boundary integral operators involved, in the framework of H\"older continuous and differentiable function spaces. A characterization of the derivative can be obtain by directly deriving the boundary values of the solution. This approach was extended to electromagnetism by Potthast \cite{Potthast3} for the perfect conductor problem  and to elasticity by Charalambopoulos \cite{Charalambopoulos} for the Dirichlet scattering problem only.  More recently, the Frechet differentiability of the class of boundary integral operators with pseudohomogeneous  hypersingular and weakly singular kernels - which includes the usual boundary integral operators occuring in time-harmonic potential theory -  was analyzed by Costabel and Le Lou\"er \cite{CostabelLeLouer, CostabelLeLouer2,FLL}, in the framework of Sobolev spaces. The analyticity of the  integral operators with respect to the boundary is proven. As a consequence, it yields the possibility to establish the Fr\'echet differentiability of the far-field pattern for any scattering problem by the use of boundary integral representations and the chain and product rules.  In this way we obtain an additional implementable formula to compute the Fr\'echet derivatives of the far field by deriving the boundary integral operators (see \cite{IvanyshynJohansson,IvanyshynKress, KressRundell}).

  This paper is devoted to the Fr\'echet differentiability  analysis of  the far-field pattern of the solution to  elastic  obstacle  scattering problems in   three-dimensional homogeneous and isotropic media via an alternative technique introduced by Kress and Pa\"{i}varinta in \cite{KressPaivarinta} to establish Fr\'echet differentiability for sound-soft and sound-hard obstacles. It is based on repeated uses of Green's theorem and  a factorization of the difference of  the far-field pattern of the scattered wave for a fixed obstacle and a perturbed obstacle. An interesting feature of the method is that it only requires the continuous dependence of the boundary values of the solution on  the boundary in order to prove the Fr\'echet differentiabilty of the far-field pattern.
  This approach was extended to the perfect conductor problem by Kress in \cite{Kress} and  the impedance problem both in acoustic and electromagnetic scattering by  Haddar and Kress in \cite{HaddarKress}.   

  The paper is organized as follow : 
   In section 2 we recall elementary results on  time-harmonic Navier equations in Sobolev Spaces, following the notations of \cite{AlvesKress}. More details can be found in \cite{CostabelStephan,HsiaoWendland}.
  The  far field identity for elastic waves in the case of a Dirichlet boundary condition was established by Alves and Kress in \cite{AlvesKress}. In section 3, we use this identity to give an alternative proof of the differentiability result due to Charalambopoulos and we improve the boundary condition satisfied by the Fr\'echet derivative of the solution. In section 4 we apply this method - following ideas of Haddar and Kress - to  establish  the Fr\'echet differentiability of the boundary to far field operator simultaneously for the cases of   Neumann and impedance boundary conditions and again we provide a characterization of the derivative.

\section{The Navier equation}
The propagation of time-harmonic elastic waves in the three-dimensional isotropic and homogeneous elastic medium characterized by the positive Lam\'e constants $\mu$ and $\lambda$ and the density $\rho$ is described by the Navier equation
\begin{equation}\label{NE}\Div\upsigma(\uu)+ \rho\omega^2\uu=0,\end{equation}
where $\omega>0$ is the frequency. Here $$\upvarsigma(\uu)=\lambda(\Div\uu)\Id_{3}+2\mu\upvarepsilon(\uu)\quad\text{and}\quad\upvarepsilon(\uu)=\frac{1}{2}\big([\nabla\uu]+\transposee{[\nabla\uu]}\big)$$
denote the stress tensor and the strain tensor respectively. Notice that  $\Id_{3}$ is the 3-by-3 identity matrix  and $[\nabla\uu]$ is the matrix whose the $j$-th column is the gradient of the $j$-th component of $\uu$. We  set $\Updelta^*\uu:=\Div\upsigma(\uu)=\mu\Delta\uu+(\lambda+\mu)\nabla\Div\uu$.

Let $\Omega\subset\R^3$ be a bounded domain with a boundary $\Gamma$ of class $\mathscr{C}^2$ and outward unit normal vector $\nn$ and let $\Omega^c$ denote the exterior domain $\R^3\backslash\overline{\Omega}$.
We denote by $H^s(\Omega)$, $H^s_{loc}(\overline{\Omega^c})$ and $H^s(\Gamma)$ the standard (local in the case of the exterior domain) complex valued, Hilbertian Sobolev space of order $s\in\R$ defined on $\Omega$, $\overline{\Omega^c}$ and $\Gamma$ respectively (with the convention $H^0=L^2$.)  Spaces of vector functions will be denoted by boldface letters, thus $\HH^s=(H^s)^3$.   We set :
\begin{eqnarray*}\HH(\Omega,\Updelta^*)&:=&\left\{\uu\in \HH^1(\Omega):\;\Updelta^*\uu\in \LL^2(\Omega)\right\},\\\HH_{loc}(\Omega^c,\Updelta^*)&:=&\left\{\uu\in \HH^1_{loc}(\overline{\Omega^c}):\;\Updelta^*\uu\in \LL_{loc}^2(\overline{\Omega^c})\right\}.\end{eqnarray*}
The space $\HH(\Omega, \Updelta^{*})$ is an Hilbert space endowed with the natural graph norm.

We use the following traces and tangential derivatives : \begin{align*}
\frac{\partial}{\partial\nn}&=\nn\cdot\nabla\textrm{ (normal derivative)},\\ T(\nn,\partial)&=2\mu\frac{\partial}{\partial\nn}+\lambda\,\nn\Div+\mu\,\nn\times\Rot \textrm{ (traction derivative)},\\\mathcal{M}(\nn,\partial)&=\frac{\partial}{\partial\nn}-\nn\Div+\,\nn\times\Rot \textrm{ (tangential G\"unter's derivative)}.\end{align*}
The tangential gradient $\nabla_{\Gamma}$ and the surface divergence $\Div_{\Gamma}$ are defined for a scalar function $u$ and a vector function $\vv$ by the following equalities \cite{Nedelec}: 
$$\nabla u =\nabla_{\Gamma}u+\frac{\partial u}{\partial\nn}\,\nn,\quad \Div\vv=\Div_{\Gamma}\vv+\;\left(\nn\cdot\frac{\partial\vv}{\partial\nn}\right),$$
and the tangential G\"unter's derivative can be rewritten as follow: \begin{equation}\label{Guntertg}\mathcal{M}\vv=[\nabla_{\Gamma}\vv]\nn-(\Div_{\Gamma}\vv)\nn.\end{equation}
We note that, due to the trace lemma, $\uu_{|\Gamma}\in\HH^{\frac{1}{2}}(\Gamma)$ for $\uu\in\HH(\Omega,\Updelta^*)\cup\HH_{loc}(\Omega^c,\Updelta^*)$. The normal derivative $\frac{\partial}{\partial\nn}\uu_{|\Gamma}$ and the traction derivative $T\uu_{|\Gamma}$ are both defined as distributions in $\HH^{-\frac{1}{2}}(\Gamma)$ via the first Green formula (see \cite{CostabelStephan, Kupradze, Nedelec} and lemma \ref{Green}).

For two $(3\times3)$ matrices $A$ and $B$ whose columns are denoted by $(a_{1},a_{2},a_{3})$ and $(b_{1},b_{2},b_{3})$, respectively, we set $A:B= a_{1}\cdot b_{1}+a_{2}\cdot b_{2}+a_{3}\cdot b_{3}$. The following lemma is a consequence of the Gauss divergence theorem and the identity \begin{equation}\label{idtt1}\Div\big(\upsigma(\uu)\cdot\vv\big)=\Updelta^*\uu\cdot\vv+\upsigma(\uu):\upvarepsilon(\vv).\end{equation}

\begin{lemma} \label{Green} For  vector functions $\uu$ and $\vv$  in $\HH(\Omega,\Delta^*)$,  it holds the first Green formula \begin{equation}\label{IPP1}\int_{\Omega}\Updelta^*\uu\cdot\vv \,dx=\int_{\Gamma}T\uu\cdot \vv\, ds-\int_{\Omega}\upsigma(\uu):\upvarepsilon(\vv)\,dx.\end{equation}
 The symmetry of the product $\upsigma(\uu):\upvarepsilon(\vv)=\lambda(\Div\uu)\cdot(\Div\vv)+2\mu\,\upvarepsilon(\uu):\upvarepsilon(\vv)=\upsigma(\vv):\upvarepsilon(\uu)$ yields the second Green formula 
\begin{equation}\label{IPP2}\int_{\Omega}\left(\uu\cdot\Updelta^*\vv-\Updelta^*\uu\cdot\vv\right)dx=\int_{\Gamma}\big(\uu\cdot T\vv-T\uu\cdot\vv\big)ds.\end{equation}
If $\uu$ and $\vv$ solve the Navier equation in $\Omega$ then each term in \eqref{IPP2} vanish.\end{lemma}
  Further we will use the following different representations of the traction operator $T$ :
\begin{eqnarray}\label{Tn1}T\uu=\upsigma(\uu)\cdot\nn&=&2\mu \mathcal{M}\uu+(\lambda+2\mu)(\Div\uu)\nn-\mu\nn\times\Rot\uu\\\label{Tn2}&=&\mu \left(\frac{\partial \uu}{\partial\nn}+\mathcal{M}\uu\right)+(\lambda+\mu)(\Div\uu)\nn\\\label{Tn3}&=&(\lambda+2\mu)\frac{\partial\uu}{\partial\nn}-\lambda\mathcal{M}\uu+(\lambda+\mu)\nn\times\Rot\uu.
\end{eqnarray}

Now we assume that the domain $\Omega$ has a connected boundary $\Gamma$.  In the sequel we are concerned with the following exterior boundary value problems for elastic waves: Given vector densities $f\in \HH^{\frac{1}{2}}(\Gamma)$ and $g\in\HH^{-\frac{1}{2}}(\Gamma)$, find a solution $\uu\in \HH_{loc}(\Omega^c,\Updelta^*)$ to the Navier equation \eqref{NE} in $\Omega^c$ which satisfies either a Dirichlet boundary condition 
\begin{equation}\label{dirichlet}\uu=f\quad\text{on }\Gamma\end{equation}
 or an impedance boundary condition 
\begin{equation}\label{impedance}T\uu+i\alpha\omega\sqrt{\rho} \, \uu=g\quad\text{on }\Gamma.\end{equation}The impedance coefficient $\alpha$ is assumed to be a real non negative constant. The case $\alpha=0$ yields the Neumann boundary condition. In addition the field $\uu$ has to satisfy the Kupradze radiation condition 
\begin{equation*}\label{radiation}\lim_{r\to\infty}r\left(\frac{\partial \uu_{p}}{\partial r}-i\kappa_{p}\uu_{p}\right)=0,\qquad\lim_{r\to\infty}r\left(\frac{\partial \uu_{s}}{\partial r}-i\kappa_{s}\uu_{s}\right)=0,\quad r=|x|,\end{equation*}
uniformly in all directions. Here, the longitudinal wave is given by $\uu_{p}=-\kappa^{-2}_{p}\nabla\Div \uu$ and  the transversal wave is given by $\uu_{s}=\uu-\uu_{p}$ associated with the respective wave numbers $\kappa_{p}$ and $\kappa_{s}$ given by
$$\kappa_{p}^2=\frac{\rho\omega^2}{\lambda+2\mu},\quad\kappa_{s}^2=\frac{\rho\omega^2}{\mu}.$$ 
Solutions of the Navier equation satisfying the Kupradze radiation condition are called radiating solution.

The fundamental solution of the Navier equation is given by $$\Phi(x,y)=\frac{1}{\mu}\left(\frac{e^{i\kappa_{s}|x-y|}}{4\pi|z|}\cdot\Id_{\R^3}+\frac{1}{\kappa_{s}^2}\nabla_{x}\transposee{\nabla}_{x}\left(\frac{e^{i\kappa_{s}|x-y|}}{4\pi|x-y|}-\frac{e^{i\kappa_{p}|x-y|}}{4\pi|x-y|}\right)\right).$$
It is a $3\times3$ matrix and for $j=1,2,3$ we denote by $\Phi_{j}$ the $j$-th column of $\Phi$. We have $\Phi(x,y)=\transposee{\Phi(x,y)}=\Phi(y,x)$.
From the second integral theorem \eqref{IPP2}, for a radiating solution $\uu\in\HH_{loc}(\Omega^c,\Updelta^*)$ to the Navier equation \eqref{NE}, one can derive the Somigliana integral representation formula for $x\in\Omega^c$:
\begin{equation}\label{intrep}\uu(x)=\int_{\Gamma}\left(\transposee{\left[T_{y}\Phi(x,y)\right]}\uu(y)-\Phi(x,y)T_{y}\uu(y)\right)ds(y),\end{equation}
where $T_{y}=T(\nn(y),\partial_{y})$ and $T_{y}\Phi(x,y)$ is the tensor obtained by applying the traction operator $T_{y}$ to each column of $\Phi(x,y)$.
For existence and uniqueness of a solution to the above boundary-value problems via boundary integral equation we refer to Kupradze \cite{Kupradze}.
The radiation condition implies that the solution  has an asymptotic behavior of the form
$$\uu(x)=\frac{e^{i\kappa_{p}|x|}}{|x|}\uu_{p}^{\infty}(\hat{x})+\frac{e^{i\kappa_{s}|x|}}{|x|}\uu_{s}^{\infty}(\hat{x})+O\left(\frac{1}{|x|}\right),\quad |x|\rightarrow\infty,$$
uniformly in all directions $\hat{x}=\dfrac{x}{|x|}$. The fields $\uu_{p}^{\infty}$ and $\uu^{\infty}_{s}$ are defined on the unit sphere $S^2$ in $\R^3$ and known as the longitudinal and the transversal far-field pattern, respectively. We introduce the $\LL^2$-spaces
\begin{eqnarray*}\LL^2_{s}(S^2)&=&\{\hh\in\LL^2(S^2);\;\hh(\hat{x})\cdot\hat{x}=0\},\\
\LL^2_{p}(S^2)&=&\{\hh\in\LL^2(S^2);\;\hh(\hat{x})\times\hat{x}=0\}.
\end{eqnarray*}
We have $\uu^{\infty}_{s}\in\LL^2_{s}(S^2)$ and $\uu^{\infty}_{p}\in\LL^2_{p}(S^2)$.

\section{The exterior  Dirichlet boundary value problem}\label{Dirichlet}
The scattering problem of time-harmonic waves by a bounded obstacle $\Omega$ leads to special cases of the above boundary value problems. In this section we consider the rigid body problem. The total displacement field $\uu+\uu^{i}$ is given by the superposition of the incident field $\uu^{i}$, which we assume to be an entire solution of the Navier equation, and the scattered field $\uu$, which solves the Navier equation in $\Omega^c$, the Dirichlet boundary  condition  $$\uu+\uu^{i}=0\quad\text{on }\Gamma,$$ and satisfies the Kupradze radiation condition. 

For $x\in\Omega^c$ let $W$ be the $3\times3$ matrix whose the $j$-th column $W_{j}$ is the radiating solution of \eqref{NE} and \eqref{dirichlet} for the boundary value
$$f=-\Phi_{j}(x,\cdot)\quad\text{on }\Gamma,$$
and set $V(x,\cdot)=\Phi(x,\cdot)+W(x,\cdot)$, that is $W$ and $V$ are the scattered and the total field, respectively, for the scattering of a point source located at $x\in\Omega^c$. We note that $V$ satisfies the reciprocity
$$V(x,y)=\transposee{[V(y,x)]},\quad x,y\in\Omega^c, x\not=y,$$
which can be derived from the second Green formula \eqref{IPP2}, the Somigliana integral representation formula \eqref{intrep} and the symmetry of the fundamental solution. 

\begin{lemma}\label{intrepdir} The unique radiating solution $\vv\in\HH_{loc}(\Omega^c,\Updelta^*)$  of the Navier equation \eqref{NE} satisfying the boundary condition \eqref{dirichlet} for any $f\in\HH^{\frac{1}{2}}(\Gamma)$ admits the following integral representation
\begin{equation}\vv(x)=\int_{\Gamma}\transposee{\left[T_{y}V(x,y)\right]}f(y)ds(y),\quad x\in\Omega^c.
\end{equation}
\end{lemma}
\begin{proof}From the second Green formula \eqref{IPP2} on $\Omega^c$ for the radiating solutions $W$ and $\vv$ we can write
$$\int_{\Gamma}\left(\transposee{\left[T_{y}W(x,y)\right]}\vv(y)-\transposee{[W(x,y)]}T_{y}\vv(y)\right)ds(y)=0,$$
for all $x\in\Omega^c$.
Using the boundary condition for $W$ and $\vv$, the symmetry of $\Phi$ and the Somigliana integral representation formula for $\vv$ we obtain
\begin{eqnarray*}\label{intrep}\vv(x)&=&\int_{\Gamma}\left(\transposee{\left[T_{y}\Phi(x,y)\right]}\vv(y)+\transposee{[W(x,y)]}T_{y}\vv(y)\right)ds(y)\\&=&\int_{\Gamma}\left(\transposee{\left[T_{y}\Phi(x,y)\right]}\vv(y)+\transposee{\left[T_{y}W(x,y)\right]}\vv(y)\right)ds(y)\\&=&\int_{\Gamma}\transposee{\left[T_{y}V(x,y)\right]}\vv(y)ds(y)=\int_{\Gamma}\transposee{\left[T_{y}V(x,y)\right]}f(y)ds(y).\end{eqnarray*}
\end{proof}

For a fixed incident field $\uu^{i}$, we consider the boundary to far field operator $$F:\Gamma\mapsto \uu^{\infty}=(\uu^{\infty}_{s},\uu^{\infty}_{p})\in\LL_{s}^2(S^2)\times\LL^2_{p}(S^2)$$ which maps the boundary of the rigid  scatterer $\Omega$ onto the far-field patterns $\uu^{\infty}_{s}$ and $\uu^{\infty}_{p}$ of the scattered field $\uu$. In order to describe the dependence of the operator $F$ on the shape of the boundary $\Gamma$, we choose a fixed reference domain $\Omega$ and we consider  variations generated by transformations of the form $$x\mapsto x+\theta(x)$$ of point $x$ in the space $\R^3$, where $\theta$ is a smooth vector function defined in the neighborhood of $\Gamma$. The functions $\theta$ are assumed to be  sufficiently small  elements of the Banach space $\mathscr{C}^{2}(\Gamma,\R^3)$ in order that $(\Id+\theta)$ is a diffeomorphism from $\Gamma$ to  $$\Gamma_{\theta}=(\Id+\theta)\Gamma=\left\{x+\theta(x); \;x\in\Gamma\right\},$$ so that the surface $\Gamma_{\theta}$ is a connected boundary of class $\mathscr{C}^2$ of a domain $\Omega_{\theta}$. The mapping $\mathcal{F}:\theta\mapsto F(\Gamma_{\theta})$ is well defined in the neighborhood of the zero function in $\mathscr{C}^2(\Gamma,\R^3)$. We then analyze the Fr\'echet differentiability of $\mathcal{F}$ at $\theta=0$.
We want to prove the existence of a linear and continuous mapping $\mathcal{F}'(0):\mathscr{C}^2(\Gamma)\rightarrow\LL_{s}^2(S^2)\times\LL^2_{p}(S^2)$ such that we have the following expansion in $\LL_{s}^2\times\LL^2_{p}$
$$\mathcal{F}(\xi)-\mathcal{F}(0)=\mathcal{F}'(0)\xi +o(||\xi||_{\mathscr{C}^2}),\quad\text{when }||\xi||_{\mathscr{C}^2}\to0.$$
To this end, since we can interchange the differentiation with respect to the boundary and the passing to the limit $|x|\to\infty$, we will establish the Fr\'echet differentiability of the scattered field away from the boundary $\Gamma$. 


By $\nn_{\theta}$ we denote the exterior unit normal vector to $\Gamma_{\theta}$ and, in what follows, we will distinguish the quantities related to the exterior Dirichlet scattering problem for the domain $\Omega_{\theta}$ through the subscript $\theta$.
We use the following identity established by Alves and Kress  in \cite{AlvesKress} pp. 13.
\begin{lemma}\label{identity}Assume that $\overline{\Omega}\subset\Omega_{\theta}$. Then
\begin{equation}\label{id}\uu_{\theta}(x)-\uu(x)=-\int_{\Gamma_{\theta}}\transposee{[V(x,y)]}T_{y}\big(\uu_{\theta}(y)+\uu^{i}(y)\big)ds(y),\end{equation}
for all $x\in\Omega_{\theta}^c$.
\end{lemma}
\begin{proof} From the second Green formula \eqref{IPP2}  for $\Phi(x,\cdot)$ and $\uu^{i}$ we can write
$$\int_{\Gamma}\left(\transposee{\left[T_{y}\Phi(x,y)\right]}\uu^{i}(y)-\Phi(x,y)T_{y}\uu^{i}(y)\right)ds(y)=0,$$
for all $x\in\Omega^c$. Using the boundary condition of $\uu$ we then obtain
\begin{equation}\label{rep1}\uu(x)=-\int_{\Gamma}\Phi(x,y)T_{y}\big(\uu(y)+\uu^{i}(y)\big)ds(y),\quad x\in\Omega^c.\end{equation}
From the second Green formula \eqref{IPP2}  for the radiating solutions $W$ and $\uu_{\theta}$ in $\Omega_{\theta}^c$ we can write
$$\int_{\Gamma_{\theta}}\left(\transposee{\left[T_{y}W(x,y)\right]}\uu_{\theta}(y)-\transposee{[W(x,y)]}T_{y}\uu_{\theta}(y)\right)ds(y)=0,\quad x\in\Omega_{\theta}^c.$$
From the second Green formula \eqref{IPP2} on $\Omega_{\theta}\backslash\overline{\Omega}$ for $W$ and $\uu^{i}$, the last equation  and the boundary condition of $\uu_{\theta}$ we have
\begin{equation}\label{aa}\hspace{-2mm}\begin{array}{l}\displaystyle{\int_{\Gamma_{\theta}}\transposee{[W(x,y)]}T_{y}\big(\uu_{\theta}(y)+\uu^{i}(y)\big)ds(y)}\\\hspace{0.3cm}=\displaystyle{\int_{\Gamma_{\theta}}\left(\transposee{[W(x,y)]}T_{y}\big(\uu_{\theta}(y)+\uu^{i}(y)\big)-\transposee{\left[T_{y}W(x,y)\right]}\big(\uu_{\theta}(y)+\uu^{i}(y)\big)\right)ds(y)}\vspace{1mm}\\\hspace{0.4cm}=\displaystyle{\int_{\Gamma_{\theta}}\left(\transposee{[W(x,y)]}T_{y}\uu^{i}(y)-\transposee{\left[T_{y}W(x,y)\right]}\uu^{i}(y)\right)ds(y)}\vspace{1mm}\\\hspace{0.4cm}=\displaystyle{\int_{\Gamma}\left(\transposee{[W(x,y)]}T_{y}\uu^{i}(y)-\transposee{\left[T_{y}W(x,y)\right]}\uu^{i}(y)\right)ds(y)},\end{array}\end{equation}
for $x\in\Omega^c_{\theta}$. Using the boundary condition of $W$ and $\uu$ on $\Gamma$, the integral representation \eqref{rep1} of $\uu$ and the second Green formula for radiating solutions we obtain
\begin{eqnarray*}-\uu(x)&=&-\int_{\Gamma}W(x,y)T_{y}\big(\uu(y)+\uu^{i}(y)\big)ds(y)\\&=&-\int_{\Gamma}\left(\transposee{[W(x,y)]}T_{y}\uu^{i}(y)+\transposee{[T_{y}W(x,y)]}\uu(y)\right)ds(y)\\&=&-\int_{\Gamma}\left(\transposee{[W(x,y)]}T_{y}\uu^{i}(y)-\transposee{[T_{y}W(x,y)]}\uu^{i}(y)\right)ds(y)\\&=&-\displaystyle{\int_{\Gamma_{\theta}}\transposee{[W(x,y)]}T_{y}\big(\uu_{\theta}(y)+\uu^{i}(y)\big)ds(y)},\end{eqnarray*}
for $x\in\Omega^c$. From $\eqref{rep1}$ we can write for $\uu_{\theta}$
\begin{equation}\label{adjintrep}\uu_{\theta}(x)=-\int_{\Gamma_{\theta}}\Phi(x,y)T_{y}\big(\uu_{\theta}(y)+\uu^{i}(y)\big)ds(y),\quad x\in\Omega_{\theta}^c.\end{equation}
We obtain the identity \eqref{id} by combining the last two equations.
\end{proof}

\begin{remark}\label{remark}As Kress and P\"aiv\"arinta pointed out in \cite{KressPaivarinta} in the acoustic case,  the lemma \ref{identity} remains valid when the domain $\Omega$ is not strictly contained in $\Omega_{\theta}$ and $W$ can be extended as a solution to the Navier equation in the exterior of $\Omega_{\theta}$. By theorem 5.7.1' in \cite{Morrey} pp. 169, this can be assured if $\Gamma$ is analytic and $\Omega_{\theta}$ does not differ too much from $\Omega$. In this case the last equality in \eqref{aa} follows by choosing an open domain $D$ such that $\overline{\Omega} \cap \Omega_{\theta}\subset D$ and then applying Green's integral theorem first in $D\backslash\Omega$ and then in $D\backslash\Omega_{\theta}$.\end{remark}

\begin{lemma}\label{identity2}Assume that $\Gamma$ is analytic. Then the following expansion holds
\begin{equation}\uu_{\theta}-\uu=\int_{\Gamma}\transposee{\left[T_{y}V(\cdot,y)\right]}B\uu(y)ds(y)+o\big(||\theta||_{\mathscr{C}^2}\big),\end{equation}
  in  $\HH(G,\Updelta^*)$ for all compact subset $G$ of $\Omega^c$ and $\theta$ sufficiently small, where
$$B\uu=-(\theta\cdot\nn)\left(\frac{1}{\mu}\big(\nn\times T(\uu+\uu^{i})\big)\times\nn+\frac{1}{\lambda+2\mu}\big(\nn\cdot T(\uu+\uu^{i})\big)\nn\right).$$
\end{lemma}
\begin{proof} We use similar arguments as in the proof of theorem 3.1 in \cite{KressPaivarinta} for the analogous acoustic case. We denote by $S_{\theta}$ and $K_{\theta}'$ the  integral operators on the boundary $\Gamma_{\theta}$ with singular kernels $2\Phi(x,y)$ and $2[T_{x}\Phi(x,y)]$ respectively. The fundamental solution $\Phi$ is pseudo-homogeneous of class $-1$. It can be shown that these operators are bounded from $\HH^{-\frac{1}{2}}(\Gamma_{\theta})$ to itself (see \cite{CostabelStephan} and \cite{Nedelec}, pp. 176). From \eqref{adjintrep} and the jump relations, it can be deduced that the traction derivative of the total field $(\uu_{\theta}+\uu^{i})$ solves the boundary integral equation
$$\left(\Id+K'_{\theta}-i\eta S_{\theta}\right)T(\uu_{\theta}+\uu^{i}_{|\Gamma_{\theta}})=2(T\uu^{i}_{|\Gamma_{\theta}}-i\eta\uu^{i}_{|\Gamma_{\theta}}).$$
The operator $K_{\theta}'$ is not compact, therefore we use a regularization technique (see \cite{Kupradze}) to modify the integral equation as below 
$$(Id+\mathcal{B}_{\theta})T(\uu_{\theta}+\uu^{i}_{|\Gamma_{\theta}})=2\mathcal{H}_{\theta}(T\uu^{i}_{|\Gamma_{\theta}}-i\eta\uu^{i}_{|\Gamma_{\theta}})$$
where $(Id+\mathcal{B}_{\theta})=\mathcal{H}_{\theta}\left(\Id+K'_{\theta}-i\eta S_{\theta}\right)$, $\mathcal{H}_{\theta}$ is a strongly integral operator and $\mathcal{B}_{\theta}$ is a weakly singular operator. The new integral operator $\Id+\mathcal{B}_{\theta}:\HH^{-\frac{1}{2}}(\Gamma_{\theta})\rightarrow\HH^{-\frac{1}{2}}(\Gamma_{\theta})$ is a Fredholm operator of the second kind which is invertible with bounded inverse.
We use  the transformation $\tau_{\theta}$  which maps a function $\uu_{\theta}$ defined on $\Gamma_{\theta}$ onto the function $\uu_{\theta}\circ(\Id+\theta)$ defined on $\Gamma$.
 Reducing the analysis in \cite{CostabelLeLouer} to the continuity and not the differentiability, we can prove that the boundary integral operators $\tau_{\theta}\mathcal{B}_{\theta}\tau_{\theta}^{-1}$ and $\tau_{\theta}\mathcal{H}_{\theta}\tau_{\theta}^{-1}$ depend continuously on the deformation $\theta\in\mathscr{C}^{2}(\Gamma)$ and so does the inverse $\tau_{\theta}(\Id+\mathcal{B}_{\theta})^{-1}\tau_{\theta}^{-1}$ from the Neumann series. Since the incident field is analytic on the boundary $\Gamma_{\theta}$, we then deduce that the total field satifies
\begin{equation}\label{aprox0}||\tau_{\theta}\left(T(\uu_{\theta}+\uu^{i})_{|\Gamma_{\theta}}\right)-T(\uu+\uu^{i})||_{\HH^{-\frac{1}{2}}(\Gamma)}\to0,\quad ||\theta||_{\mathscr{C}^2}\to0\end{equation}
Since $\Gamma$ is analytic, by theorem 5.7.1' in \cite{Morrey} pp. 169 the total fields $(\uu+\uu^{i})$ and $W$ can be extended as a solution to the Navier equation across $\Gamma$ in the exterior of $\overline{\Omega}\cap\Omega_{\theta}$. We also have
\begin{equation}\label{aprox1}||\tau_{\theta}\left(T(\uu+\uu^{i})_{|\Gamma_{\theta}}\right)-T(\uu+\uu^{i})||_{\HH^{-\frac{1}{2}}(\Gamma)}\to0,\quad ||\theta||_{\mathscr{C}^2}\to0.\end{equation}
 By lemma \ref{identity} and  using Taylor's formula together  with \eqref{aprox0}-\eqref{aprox1}, the boundary condition for $V$ and the smoothness of $V(x,\cdot)$  up to $\Gamma$ for $x$ away from the boundary, it follows that
$$\int_{\Gamma_{\theta}}\transposee{[V(\cdot,y)]}T_{y}\big(\uu_{\theta}(y)+\uu^{i}(y)\big)ds(y)=\int_{\Gamma_{\theta}}\transposee{[V(\cdot,y)]}T_{y}\big(\uu(y)+\uu^{i}(y)\big)ds(y)+o(||\theta||_{\mathscr{C}^2}),$$
 in  $\HH(G,\Updelta^*)$ for all compact subset $G$ of $\Omega^c$.
 From this, the first Green formula \eqref{IPP1} and the boundary condition for $W$ we obtain
\begin{equation*}\uu_{\theta}-\uu=-\int_{\Omega_{\theta}^{*}}\left\{\Big(\big[\nabla V_{j}\big]:\upsigma\big(\uu+\uu^{i}\big)\Big)_{j=1,2,3}-\rho\omega^2\transposee{[V]}(\uu+\uu^{i})\right\}\chi dy+o(||\theta||_{\mathscr{C}^2})\end{equation*}
where $$\Omega_{\theta}^*=\{y\in\Omega_{\theta};y\not\in\Omega\}\cup\{y\in\Omega;y\not\in\Omega_{\theta}\},$$
and $\chi(y)=1$ if $y\in\Omega_{\theta}$ and $y\not\in\Omega$ and $\chi(y)=-1$ if $y\in\Omega$ and $y\not\in\Omega_{\theta}$. 
Any $z\in\Omega_{\theta}^*$ can be represented of the form $z=y+t\theta(y)$, with $y\in\Gamma$ and $t>0$. We have :
\begin{itemize}
\item[(a)] $\chi(z)dz=\big(\theta(y)\cdot\nn(y)\big)ds(y)dt+o(||\theta||_{\mathscr{C}^2})$,
\item[(b)] $(\uu+\uu^{i})(y+t\theta(y))=o(||\theta||_{\mathscr{C}^2})$ and $V(\cdot,y+t\theta(y))=o(||\theta||_{\mathscr{C}^2})$,
\item[(c)] $\nabla\big(\uu+\uu^{i}\big)=\dfrac{\partial}{\partial\nn}(\uu+\uu^{i})\cdot\transposee{\nn}+o(||\theta||_{\mathscr{C}^2})$ and,
\item[(d)] for $j=1,2,3$, $$\big[\nabla V_{j}\big]:\upsigma\big(\uu+\uu^{i}\big)=\upsigma\big(V_{j}\big):\big[\nabla(\uu+\uu^{i})\big]=T_{y}V_{j}\cdot\dfrac{\partial}{\partial\nn}(\uu+\uu^{i})+o(||\theta||_{\mathscr{C}^2}).$$ \end{itemize}
The expansion $(a)$ for the volume form is a well-known result and we refer to \cite{KressPaivarinta,HaddarKress, PierreHenrot} for a proof.
Approximating the integral over $\Omega_{\theta}^*$ by an integral over $\Gamma$ we obtain that
\begin{equation*}\uu_{\theta}-\uu=-\int_{\Gamma}\transposee{\left[T_{y}V(\cdot,y)\right]}\frac{\partial(\uu+\uu^{i})}{\partial\nn}(y)\big(\theta(y)\cdot\nn(y)\big)ds(y)+o(||\theta||_{\mathscr{C}^2}).\end{equation*}
To conclude, we express the normal derivative of $(\uu+\uu^{i})$ in function of its traction derivative. Since $\uu+\uu^{i}=0$ on $\Gamma$ then the tangential Gunter's derivative $\mathcal{M}(\uu+\uu^{i})_{|\Gamma}$ vanishes. 
Using \eqref{Tn2}, we obtain
$$\left(\nn\times T(\uu+\uu^{i})\right)\times\nn=\mu\left(\nn\times\frac{\partial}{\partial\nn}(\uu+\uu^{i})\right)\times\nn,$$
and using \eqref{Tn3} we obtain
$$\nn\cdot T(\uu+\uu^{i})=(\lambda+2\mu)\left(\nn\cdot\frac{\partial(\uu+\uu^{i})}{\partial\nn}\right).$$
\end{proof}

\begin{theorem}\label{Th1} Let $\Gamma$ be analytic. Then the mapping $\mathcal{F}:{\mathscr{C}^2(\Gamma,\R^3)\hspace{-.5mm}\rightarrow\hspace{-.5mm}\LL^2_{s}(S^2)\hspace{-.5mm}\times\hspace{-.5mm}\LL^2_{p}(S^2)}$ is Fr\'echet differentiable at $\theta=0$ with the  Fr\'echet derivative defined for $\xi\in\mathscr{C}^2(\Gamma,\R^3)$ by $$\mathcal{F}'(0)\xi=\vv^{\infty}_{\xi},$$
where $\vv^{\infty}_{\xi}$ is the far-field pattern of the solution $\vv_{\xi}$ to the Navier equation in $\Omega^c$ that satisfies the Kupradze radiation condition and the Dirichlet boundary condition
$$\vv_{\xi}=-(\xi\cdot\nn)\left(\frac{1}{\mu}\big(\nn\times T(\uu+\uu^{i})\big)\times\nn+\frac{1}{\lambda+2\mu}\big(\nn\cdot T(\uu+\uu^{i})\big)\nn\right).$$
\end{theorem}
\begin{proof} This is a direct consequence of Lemmas \ref{intrepdir} and \ref{identity2}.
\end{proof}

Although theorem \ref{Th1} was proven under the assumption that $\Gamma$ is analytic, we expect that, proceeding as in \cite{KressPaivarinta}, the result is also valid for $\mathscr{C}^2$ boundaries. 

\section{The exterior   impedance boundary value problem}
Now we consider the scattering problem by a cavity are an absorbing obstacle $\Omega$. The total displacement field $\uu+\uu^{i}$ is given by the superposition of the incident field $\uu^{i}$, which we assume to be an entire solution of the Navier equation, and the scattered field $\uu$, which solves the Navier equation in $\Omega^c$, the impedance boundary  condition  $$T(\uu+\uu^{i})+i\alpha\omega\sqrt{\rho}\,(\uu+\uu^{i})=0\quad\text{on }\Gamma,$$ and satisfies the Kupradze radiation condition.

For $x\in\Omega^c$ let $W$ be the $3\times3$ matrix whose the $j$-th column $W_{j}$ is the radiating solution of \eqref{NE} and \eqref{dirichlet} for the boundary value
$$g=-\left(T\Phi_{j}(x,\cdot)+i\alpha\omega\sqrt{\rho}\,\Phi_{j}(x,\cdot)\right)\quad\text{on }\Gamma,$$
and set $V(x,\cdot)=\Phi(x,\cdot)+W(x,\cdot)$.

\begin{lemma}\label{intrepimp} The unique radiating solution $\vv\in\HH_{loc}(\Omega^c,\Updelta^*)$  of the Navier equation \eqref{NE} satisfying the boundary condition \eqref{impedance} for any $g\in\HH^{-\frac{1}{2}}(\Gamma)$ admits the following integral representation
\begin{equation}\vv(x)=-\int_{\Gamma}\transposee{\left[V(x,y)\right]}g(y)ds(y),\quad x\in\Omega^c.
\end{equation}
\end{lemma}
\begin{proof}From the second Green formula \eqref{IPP2} on $\Omega^c$ for the radiating solutions $W$ and $\vv$ we can write
$$\int_{\Gamma}\left(\transposee{\left[(T_{y}+i\alpha\omega\sqrt{\rho})W(x,y)\right]}\vv(y)-\transposee{[W(x,y)]}(T_{y}+i\alpha\omega\sqrt{\rho})\vv(y)\right)ds(y)=0,$$
for all $x\in\Omega^c$.
Using the boundary condition for $W$ and $\vv$ and  the Somigliana integral representation formula for $\vv$ we obtain
\begin{eqnarray*}\label{intrep}\vv(x)&=&\int_{\Gamma}\left(\transposee{\big[(T_{y}+i\alpha\omega\sqrt{\rho})\Phi(x,y)\big]}\vv(y)-\transposee{[\Phi(x,y)]}(T_{y}+i\alpha\omega\sqrt{\rho})\vv(y)\right)ds(y)\\&=&\int_{\Gamma}\left(-\transposee{\big[(T_{y}+i\alpha\omega\sqrt{\rho})W(x,y)\big]}\vv(y)-\transposee{\left[\Phi(x,y)\right]}(T_{y}+i\alpha\omega\sqrt{\rho})\vv(y)\right)ds(y)\\&=&-\int_{\Gamma}\transposee{\left[V(x,y)\right]}(T_{y}+i\alpha\omega\sqrt{\rho})\vv(y)ds(y)=-\int_{\Gamma}\transposee{\left[V(x,y)\right]}g(y)ds(y).\end{eqnarray*}
\end{proof}
Here again we use a parametrization of the boundaries in order to investigate the Fr\'echet differentiability of  the boundary to far field operator $$F:\Gamma\mapsto \uu^{\infty}=(\uu^{\infty}_{s},\uu^{\infty}_{p})\in\LL_{s}^2(S^2)\times\LL^2_{p}(S^2)$$ which maps the boundary of the Neumann or impedance obstacle $\Omega$ onto the far-field patterns $\uu^{\infty}_{s}$ and $\uu^{\infty}_{p}$ of the scattered field $\uu$. Thus we will consider instead the mapping $\mathcal{F}:\theta\mapsto F(\Gamma_{\theta})$. The following lemma give a factorization of the difference of the scattered field for two neighboring impedance obstacles. 
\begin{lemma}\label{identity3}Assume that $\overline{\Omega}\subset\Omega_{\theta}$. Then
\begin{equation}\label{id3}\uu_{\theta}(x)-\uu(x)=\int_{\Gamma_{\theta}}\transposee{[T_{y}V(x,y)+i\alpha\omega\sqrt{\rho}\, V(x,y)]}\big(\uu_{\theta}(y)+\uu^{i}(y)\big)ds(y),\end{equation}
for all $x\in\Omega_{\theta}^c$.
\end{lemma}
\begin{proof} From the second Green formula \eqref{IPP2}  for $\Phi(x,\cdot)$ and $\uu^{i}$ we can write
$$\int_{\Gamma}\left(\transposee{\left[(T_{y}+i\alpha\omega\sqrt{\rho})\Phi(x,y)\right]}\uu^{i}(y)-\Phi(x,y)(T_{y}+i\alpha\omega\sqrt{\rho})\uu^{i}(y)\right)ds(y)=0,$$
for all $x\in\Omega^c$. Using the boundary condition of $\uu$ we then obtain
\begin{equation}\label{rep3}\uu(x)=\int_{\Gamma}\transposee{\left[(T_{y}+i\alpha\omega\sqrt{\rho})\Phi(x,y)\right]}\big(\uu(y)+\uu^{i}(y)\big)ds(y),\quad x\in\Omega^c.\end{equation}
From the second Green formula \eqref{IPP2}  for the radiating solutions $W$ and $\uu_{\theta}$  in $\Omega_{\theta}^c$ we can write for $x\in\Omega_{\theta}^c$
$$\int_{\Gamma_{\theta}}\left\{\transposee{\left[(T_{y}+i\alpha\omega\sqrt{\rho})W(x,y)\right]}\uu_{\theta}(y)-\transposee{[W(x,y)]}(T_{y}+i\alpha\omega\sqrt{\rho})\uu_{\theta}(y)\right\}ds(y)=0.$$
From the second Green formula \eqref{IPP2}  for $W$ and $\uu^{i}$ in $\Omega_{\theta}\backslash\overline{\Omega}$, the last equation  and the boundary condition of $\uu_{\theta}$ we have
\begin{equation}\label{aa}\hspace{-.1cm}\begin{array}{l}\displaystyle{\int_{\Gamma_{\theta}}\transposee{[(T_{y}+i\alpha\omega\sqrt{\rho})W(x,y)]}\big(\uu_{\theta}(y)+\uu^{i}(y)\big)ds(y)}\\\hspace{.5cm}=\displaystyle{\int_{\Gamma_{\theta}}\left(\transposee{[(T_{y}+i\alpha\omega\sqrt{\rho})W(x,y)]}\big(\uu_{\theta}(y)+\uu^{i}(y)\big)\right.}\\\left.\hspace{2cm}-\transposee{\left[W(x,y)\right]}(T_{y}+i\alpha\omega\sqrt{\rho})\big(\uu_{\theta}(y)+\uu^{i}(y)\big)\right)ds(y)\vspace{1mm}\\\hspace{.5cm}=\displaystyle{\int_{\Gamma_{\theta}}\left(\transposee{[(T_{y}+i\alpha\omega\sqrt{\rho})W(x,y)]}\uu^{i}(y)-\transposee{\left[W(x,y)\right]}(T_{y}+i\alpha\omega\sqrt{\rho})\uu^{i}(y)\right)ds(y)}\vspace{1mm}\\\hspace{.5cm}=\displaystyle{\int_{\Gamma}\left(\transposee{[(T_{y}+i\alpha\omega\sqrt{\rho})W(x,y)]}\uu^{i}(y)-\transposee{\left[W(x,y)\right]}(T_{y}+i\alpha\omega\sqrt{\rho})\uu^{i}(y)\right)ds(y)},\end{array}\end{equation}
for $x\in\Omega^c_{\theta}$. Using the boundary condition of $W$ and $\uu$ on $\Gamma$, the integral representation \eqref{rep3} of $\uu$ and the second Green formula for radiating solutions we obtain
\begin{eqnarray*}-\uu(x)&=&\int_{\Gamma}\transposee{[(T_{y}+i\alpha\omega\sqrt{\rho})W(x,y)]}\big(\uu(y)+\uu^{i}(y)\big)ds(y)\\&=&\int_{\Gamma}\left(\transposee{[(T_{y}+i\alpha\omega\sqrt{\rho})W(x,y)]}\uu^{i}(y)+\transposee{[W(x,y)]}(T_{y}+i\alpha\omega\sqrt{\rho})\uu(y)\right)ds(y)\\&=&\int_{\Gamma}\left(\transposee{[(T_{y}+i\alpha\omega\sqrt{\rho})W(x,y)]}\uu^{i}(y)-\transposee{[W(x,y)]}(T_{y}+i\alpha\omega\sqrt{\rho})\uu^{i}(y)\right)ds(y)\\&=&\displaystyle{\int_{\Gamma_{\theta}}\transposee{[(T_{y}+i\alpha\omega\sqrt{\rho})W(x,y)]}\big(\uu_{\theta}(y)+\uu^{i}(y)\big)ds(y)},\end{eqnarray*}
for $x\in\Omega^c$. From $\eqref{rep3}$ we can write for $\uu_{\theta}$
\begin{equation}\label{adjintrep2}\uu_{\theta}(x)=-\int_{\Gamma_{\theta}}\transposee{\left[(T_{y}+i\alpha\omega\sqrt{\rho})\Phi(x,y)\right]}\big(\uu_{\theta}(y)+\uu^{i}(y)\big)ds(y),\quad x\in\Omega_{\theta}^c.\end{equation}
We obtain the identity \eqref{id3} by combining the last two equations.
\end{proof}
The remark \ref{remark} is still availaible here, so that we expect that the identity \eqref{id3} remains valid when the domain $\Omega$ is not strictly contained in $\Omega_{\theta}$. 

We denote by $H_{\Gamma}$ the mean curvature of $\Gamma$  defined by $$H_{\Gamma}=\Div_{\Gamma}\nn.$$

\begin{lemma}\label{identity4}Assume that $\Gamma$ is analytic. Then the following expansion holds
\begin{equation}\uu_{\theta}(x)-\uu(x)=-\int_{\Gamma}{\left[V(x,y)\right]}B\uu(y)ds(y)+o\big(||\theta||_{\mathscr{C}^2(\Gamma,\R^3)}\big),\end{equation}
  in  $\HH(G,\Updelta^*)$ for all compact subset $G$ of $\Omega^c$ and $\theta$ sufficiently small, where
  \begin{equation}\label{carac2}\hspace{-3mm}\begin{array}{l}B\uu=\Div_{\Gamma}\Big(\hspace{-.5mm}(\theta\cdot\nn)\big\{\upsigma_{_{\rm I}}(\uu+\uu^{i})-i\alpha\omega\sqrt{\rho}\,\nn\cdot\transposee{(\uu_{t}+\uu^{i}_{t})}\big\}\hspace{-.5mm}\Big)\vspace{2mm}\\\hspace{1cm}+\,\rho\omega^2(\theta\cdot\nn)\left(\hspace{-1mm}\big(1-\dfrac{\alpha^2}{\mu}\big)(\uu_{t}+\uu^{i}_{t})+\big(1-\dfrac{\alpha^2}{\lambda+2\mu}\big)(\uu_{n}+\uu^{i}_{n})\nn\hspace{-.5mm}\right)\\\hspace{4mm}+\,i\alpha\omega\sqrt{\rho}(\theta\cdot\nn)\left(\hspace{-.5mm}[\nabla_{\Gamma}(\uu+\uu^{i})]\nn+\dfrac{\lambda}{\lambda+2\mu}(\Div_{\Gamma}(\uu+\uu^{i}))\nn-(\uu+\uu^{i})H_{\Gamma}\hspace{-.5mm}\right),\end{array}\end{equation}
with
\begin{equation}\label{sigmatau}\upsigma_{_{\rm I}}(\uu)=\lambda\left(\frac{2\mu}{\lambda+2\mu}\Div_{\Gamma}\uu-i\frac{\alpha\omega\sqrt{\rho}}{\lambda+2\mu}\uu_{n}\right)\Pi_{3}+\mu\, \Pi_{3}\left([\nabla_{\Gamma}\uu]+\transposee{[\nabla_{\Gamma}\uu]}\right)\Pi_{3},\end{equation}
and 
$$\Pi_{3}=\Id_{3}-\nn\cdot\transposee{\nn},\quad \uu_{t}=(\nn\times\uu)\times\nn\quad\text{ and }\quad\uu_{n}=\uu\cdot\nn.$$
\end{lemma}
\begin{proof}   We follow the same procedure as in the proof of lemma \ref{identity2}. We denote by $S_{\theta}$ and $K_{\theta}$ the  integral operators on the boundary $\Gamma_{\theta}$ with singular kernels $2\Phi(x,y)$ and $2\transposee{[T_{y}\Phi(x,y)]}$ respectively. Assume that $\alpha>0$. From \eqref{adjintrep2} and the jump relations, it can be deduced that the restriction to $\Gamma$ of the total field $(\uu_{\theta}+\uu^{i})$ solves the boundary integral equation
\begin{equation}\label{inteqimp}\left(\Id-K_{\theta}-i\alpha S_{\theta}\right)(\uu_{\theta}+\uu^{i})_{|\Gamma_{\theta}}=2\uu^{i}_{|\Gamma_{\theta}}.\end{equation}
The operator $K_{\theta}$ is not compact. By regularization technique (see \cite{Kupradze}), here again, we can  modify the above equation  
in order to obtain an integral equation of the second kind which has to be solved for the unkown $\uu_{\theta}+\uu^{i}$ in $\HH^{\frac{1}{2}}(\Gamma)$. From this new equation and the convergence of integral operators as $\theta\to0$, it can be deduce that the total field satifies
\begin{equation}\label{approx3}||\tau_{\theta}(\uu_{\theta}+\uu^{i})_{|\Gamma_{\theta}}-(\uu+\uu^{i})_{|\Gamma}||_{\HH^{\frac{1}{2}}(\Gamma)}\to0,\quad ||\theta||_{\mathscr{C}^2}\to0\end{equation}
Then we can write that
$$\begin{array}{ll}&\displaystyle{\int_{\Gamma_{\theta}}\transposee{[(T_{y}+i\alpha\omega\sqrt{\rho})V(\cdot,y)]}\big(\uu_{\theta}(y)+\uu^{i}(y)\big)ds(y)}\\&\hspace{2cm}=\displaystyle{\int_{\Gamma_{\theta}}\transposee{[(T_{y}+i\alpha\omega\sqrt{\rho})V(\cdot,y)]}\big(\uu(y)+\uu^{i}(y)\big)ds(y)}+o(||\theta||_{\mathscr{C}^2}),\end{array}$$
 in  $\HH(G,\Updelta^*)$ for all compact subset $G$ of $\Omega^c$. 
 Notice that the outer unit normal vector $\nn$ to $\Gamma$ can be extended in a continuously differentiable function, denoted again by $\nn$, on a tubular $$B_{T}=\{z=y+t\nn(y);\; y\in\Gamma,\; t\in[-T;T]\}$$ for some sufficiently small $T$. 
 By the first Green formula \eqref{IPP1} together with the following expansion for $y\in\Gamma$ (see \cite{HaddarKress, PierreHenrot})
 $$\nn(y)\cdot\big(\nn_{\theta}(y)-\nn(y)\big)=-\,\nn(y)\cdot\nabla_{\Gamma}\big(\theta(y)\cdot\nn(y)\big)+o(||h||_{\mathscr{C}^2})=o(||h||_{\mathscr{C}^2})$$
 and the boundary condition for $W$ we obtain
\begin{equation*}\uu_{\theta}-\uu=\int_{\Omega_{\theta}^{*}}\left\{\Div\Big(\upsigma( V_{j})(\uu+\uu^{i})+i\alpha\omega\sqrt{\rho}\,\big(V_{j}\cdot(\uu+\uu_{i})\big)\nn\Big)_{j=1,2,3}\right\}\chi dy+o(||\theta||_{\mathscr{C}^2})\end{equation*}
where $$\Omega_{\theta}^*=\{y\in\Omega_{\theta};y\not\in\Omega\}\cup\{y\in\Omega;y\not\in\Omega_{\theta}\},$$
and $\chi(y)=1$ if $y\in\Omega_{\theta}$ and $y\not\in\Omega$ and $\chi(y)=-1$ if $y\in\Omega$ and $y\not\in\Omega_{\theta}$. 
Approximating the integral over $\Omega_{\theta}^*$ by an integral over $\Gamma$ we obtain that
$$\begin{array}{l}\uu_{\theta}-\uu\\\hspace{3mm}=\displaystyle{\int_{\Gamma}\Div\Big(\upsigma( V_{j})(\uu+\uu^{i})+i\alpha\omega\sqrt{\rho}\,\big(V_{j}\cdot(\uu+\uu_{i})\big)\nn\Big)_{_{j=1,2,3}}(\theta\cdot\nn) ds(y)}+o(||\theta||_{\mathscr{C}^2})\end{array}$$
For $j=1,2,3$, we have $$\begin{array}{lcl}\Div\Big(\upsigma( V_{j})(\uu+\uu^{i})\Big)&=&\upsigma(V_{j}):\big[\nabla(\uu+\uu^{i})\big]-\rho\omega^{2}\big(V_{j}\cdot(\uu+\uu^{i})\big)\vspace{1mm}\\&=&\big[\nabla V_{j}\big]:\upsigma\big(\uu+\uu^{i}\big)-\rho\omega^{2}\big(V_{j}\cdot(\uu+\uu^{i})\big)\vspace{2mm}\\&=&[\nabla_{\Gamma}(V_{j})]:\left(\Pi_{3}\upsigma(\uu+\uu^{i})\Pi_{3}\right)-\rho\omega^{2}\big(V_{j}\cdot(\uu+\uu^{i})\big)\vspace{2mm}\\&&-i\alpha\omega\sqrt{\rho}\left(\dfrac{\partial}{\partial\nn}V_{j}\cdot (\uu+\uu^{i})+[\nabla_{\Gamma}V_{j}]:\left((\uu_{t}+\uu^{i}_{t})\cdot\transposee{\nn}\right)\right)\end{array}$$ 
and 
$$\Div\Big(\big(V_{j}\cdot(\uu+\uu_{i})\big)\nn\Big)=\left(\dfrac{\partial}{\partial\nn}+H_{\Gamma}\right)\big(V_{j}\cdot(\uu+\uu_{i})\big)$$
Collecting the two above equalities we obtain
\begin{equation}\label{bound}\begin{array}{ll}&\Div\Big(\upsigma( V_{j})(\uu+\uu^{i})+i\alpha\omega\sqrt{\rho}\,\big(V_{j}\cdot(\uu+\uu_{i})\big)\nn\Big)\vspace{2mm}\\&\hspace{2cm}=[\nabla_{\Gamma}(V_{j})]:\left(\Pi_{3}\upsigma(\uu+\uu^{i})\Pi_{3}-i\alpha\omega\sqrt{\rho}\left((\uu_{t}+\uu^{i}_{t})\cdot\transposee{\nn}\right)\right)\vspace{2mm}\\&\hspace{2cm}-\rho\omega^{2}\big(V_{j}\cdot(\uu+\uu^{i})\big)+i\alpha\omega\sqrt{\rho}V_{j}\cdot \left(\dfrac{\partial}{\partial\nn}+H_{\Gamma}\right)(\uu+\uu^{i})\end{array}\end{equation}
To conclude we have to express $\Pi_{3}\upsigma(\uu+\uu^{i})\Pi_{3}$ and $\frac{\partial}{\partial\nn}(\uu+\uu^{i})$ in function of the tangential derivatives of $\uu$.
First, we note that
$$\Pi_{3}\upsigma(\uu+\uu^{i})\Pi_{3}=\lambda\Div(\uu+\uu^{i})\Pi_{3}+\mu\, \Pi_{3}\left([\nabla_{\Gamma}\uu]+\transposee{[\nabla_{\Gamma}\uu]}\right)\Pi_{3}$$
Then we use the identity \eqref{Tn1} together with the boundary condition of $\uu$, which gives 
$$\Div(\uu+\uu^{i})=\frac{2\mu}{\lambda+2\mu}\Div_{\Gamma}(\uu+\uu^{i})-i\frac{\alpha}{\lambda+2\mu}\omega\sqrt{\rho}\,\nn\cdot(\uu+\uu^{i}).$$
The identity \eqref{Tn2} yields
$$\left(\nn\times\frac{\partial}{\partial\nn}(\uu+\uu^{i})\right)\times\nn=-[\nabla_{\Gamma}(\uu+\uu^{i})]\nn-i\frac{\alpha}{\mu}\omega\sqrt{\rho}\,\big(\nn\times(\uu+\uu^{i})\big)\times\nn,$$
end the identity \eqref{Tn3} yields
$$\nn\cdot\frac{\partial}{\partial\nn}(\uu+\uu^{i})=-\frac{\lambda}{\lambda+2\mu}\Div_{\Gamma}(\uu+\uu^{i})-i\frac{\alpha}{\lambda+2\mu}\omega\sqrt{\rho}\,\nn\cdot(\uu+\uu^{i}).$$
Substituing all the above identity in \eqref{bound} we obtain the characterization \eqref{carac2}. 
\end{proof}

\begin{theorem} Let $\Gamma$ be analytic. Then the mapping $\mathcal{F}:{\mathscr{C}^2(\Gamma,\R^3)\hspace{-.5mm}\rightarrow\hspace{-.5mm}\LL^2_{s}(S^2)\hspace{-.5mm}\times\hspace{-.5mm}\LL^2_{p}(S^2)}$ is Fr\'echet differentiable at $\theta=0$ with the  Fr\'echet derivative defined for $\xi\in\mathscr{C}^2(\Gamma,\R^3)$ by $$\mathcal{F}'(0)\xi=\vv^{\infty}_{\xi},$$
where $\vv^{\infty}_{\xi}$ is the far-field pattern of the solution $\vv_{\xi}$ to the Navier equation in $\Omega^c$ that satisfies the Kupradze radiation condition and the impedance boundary condition
$$\hspace{-4mm}\begin{array}{ll}&(T+i\alpha\omega\sqrt{\rho})\vv_{\xi}\vspace{2mm}\\&\hspace{.8cm}=\;\Div_{\Gamma}\Big((\xi\cdot\nn)\big\{\upsigma_{_{\rm I}}(\uu+\uu^{i})-i\alpha\omega\sqrt{\rho}\,\nn\cdot\transposee{(\uu_{t}+\uu^{i}_{t})}\big\}\Big)\vspace{2mm}\\&\hspace{1cm}+\;\rho\omega^2(\xi\cdot\nn)\left(\big(1-\dfrac{\alpha^2}{\mu}\big)(\uu_{t}+\uu^{i}_{t})+\big(1-\dfrac{\alpha^2}{\lambda+2\mu}\big)(\uu_{n}+\uu_{n}^{i})\nn\right)\\&\hspace{1cm}+\;i\alpha\omega\sqrt{\rho}(\xi\cdot\nn)\left([\nabla_{\Gamma}(\uu+\uu^{i})]\nn+\dfrac{\lambda}{\lambda+2\mu}\big(\Div_{\Gamma}(\uu+\uu^{i})\big)\nn-(\uu+\uu^{i})H_{\Gamma}\right),\end{array}$$
where the symmetric tensor $\upsigma_{_{\rm I}}(\uu)$ is given by \eqref{sigmatau}.
\end{theorem}
\begin{proof} This is a direct consequence of Lemmas  \ref{intrepimp} and \ref{identity4}.
\end{proof}
When $\alpha=0$, the integral equation \eqref{inteqimp} is not uniquely solvable but one can prove that the total field $(\uu_{\theta}+\uu^{i})$ is the unique solution of an hypersingular boundary integral equation. Via regularization method on can prove that the estimation \eqref{approx3} is still valid in this case.   For the Neumann problem, we then obtain the characterization:

\begin{theorem} Let $\Gamma$ be analytic. Then the mapping $\mathcal{F}:{\mathscr{C}^2(\Gamma,\R^3)\hspace{-.5mm}\rightarrow\hspace{-.5mm}\LL^2_{s}(S^2)\hspace{-.5mm}\times\hspace{-.5mm}\LL^2_{p}(S^2)}$ is Fr\'echet differentiable at $\theta=0$ with the  Fr\'echet derivative defined for $\xi\in\mathscr{C}^2(\Gamma,\R^3)$ by $$\mathcal{F}'(0)\xi=\vv^{\infty}_{\xi},$$
where $\vv^{\infty}_{\xi}$ is the far-field pattern of the solution $\vv_{\xi}$ to the Navier equation in $\Omega^c$ that satisfies the Kupradze radiation condition and the Neumann boundary condition
$$\begin{array}{ll}T\vv_{\xi}=\Div_{\Gamma}\big((\xi\cdot\nn)\upsigma_{_{\rm N}}(\uu+\uu^{i})\big)+\rho\omega^2(\xi\cdot\nn)(\uu+\uu^{i}),\end{array}$$
where
\begin{equation}\upsigma_{_{\rm N}}(\uu)=\frac{2\lambda\mu}{\lambda+2\mu}(\Div_{\Gamma}\uu)\Pi_{3}+\mu\, \Pi_{3}\left([\nabla_{\Gamma}\uu]+\transposee{[\nabla_{\Gamma}\uu]}\right)\Pi_{3}.\end{equation}
\end{theorem}

\section{Acknowlegements}
This work was supported by the project SBF 755 "Nanoscale Photonic Imaging".

   \end{document}